\documentclass[12pt]{article}
\usepackage{amsmath}
\usepackage{amsthm}
\usepackage{amscd}
\usepackage{amssymb}
\usepackage[all]{xy}
\usepackage[usenames]{color}

\usepackage{latexsym, a4}
\usepackage{graphicx}

\newcommand{\Lie}{\ensuremath{\mathsf{Lie}}}

\newtheorem{theorem}{Theorem}[section]

\newtheorem{corollary}[theorem]{Corollary}

\newtheorem{lemma}[theorem]{Lemma}

\newtheorem{proposition}[theorem]{Proposition}

\theoremstyle{definition}

\newtheorem{definition}[theorem]{Definition}
\newtheorem{remark}[theorem]{Remark}
\newtheorem{example}[theorem]{Example}

\newcommand{\Ker}{{\sf Ker}}

\title{\bf {\Lie}-isoclinism of pairs of Leibniz algebras}

\author{Z. Riyahi$^{a, }$\footnote{
E-mail: z.riyahi@mazust.ac.ir} ~~ and ~~ J. M. Casas$^{b }$\footnote{E-mail: jmcasas@uvigo.es}\\
{\small $^a$ University of Sciences and Technology of Mazandaran, Behshahr, Iran}\\
{\small $^b$  Dpto. Matem\'atica Aplicada I, Universidade de Vigo, E. E. Forestal,}\\
{\small Campus Universitario A Xunqueira, 36005 Pontevedra, Spain}
\date{}
}

\begin{document}
\maketitle

\begin{abstract}
The aim of this paper is to consider the relation between {\Lie}-isoclinism and isomorphism of two pairs of Leibniz algebras. We show that, unlike the absolute case for finite dimensional Lie algebras, these concepts are not identical, even if the pairs of Leibniz algebras are {\Lie}-stem. Moreover, throughout the paper we provide some conditions under which {\Lie}-isoclinism and isomorphism of {\Lie}-stem Leibniz algebras are equal. In order to get this equality, the concept of factor set is studied as well.\\

\noindent {\it Keywords:} Leibniz algebras, {\Lie}-isoclinic pairs, {\Lie}-stem pair, Factor set.\\
\noindent {\it Mathematics Subject Classification} 2010: 17A32, 18B99.
\end{abstract}

\section{Introduction}
\label{intro}
The isoclinism in group theory, that is an equivalence relation on groups which generalizes isomorphism, was first introduced by Hall \cite{RefH} for the purpose of classifying finite p-groups of small order. This concept was studied by several authors, including  Tappe \cite{RefT} and Weichsel \cite{RefW}.
In 1994,  Moneyhun \cite{RefM} extended this concept to Lie algebras that produces a partition on the class of all Lie algebras into equivalences classes. By this equivalence relation, she showed that the isoclinic family of Lie algebras contains at
least one stem Lie algebra. Also, she proved that the concepts of isoclinism and isomorphism between Lie algebras of the same finite dimension are identical. 
The isoclinism of a pair of Lie algebras was studied by Moghaddam  et al. \cite{RefMP} in $ 2009 $. They generalised the first result of  Moneyhun for the pair of Lie algebras. In addition, it showed that two pairs of finite dimensional stem Lie algebras are isoclinic if and only if they are isomorphic. A passage to the similar to the above, for the central extension of Lie algebras (in \cite{RefMSR}) present, under some conditions, the same result.

In the last decades, a prominent research line consists in the extension of properties from Lie algebras to Leibniz algebras, which are non-anti-commutative versions of Lie algebras \cite{RefL,RefLP}. In more detail, a vector space $ \mathfrak{q} $ equipped with a bilinear map $ [-, -] : \mathfrak{q} \times \mathfrak{q}\rightarrow \mathfrak{q}$ is called Leibniz algebra if satisfying the Leibniz identity: \[[x, [y, z]]=[[x, y], z] - [[x, z], y],~~~~x, y, z\in \mathfrak{q}.\]

The investigations on Leibniz algebras theory show that some results of the theory of Lie algebras can be extended to Leibniz one.
It is of interest to know whether the above mentioned works, in particular, the equivalence between isoclinism and isomorphism in presence of finite dimension, \cite{RefMP,RefM}, are still true for the Leibniz algebras. So the main goal of this paper is to answer this question, for that we focus in the relative framework, that is the context relative to the Liezation functor as we explain below.

In the papers \cite{RefBC,RefCKh} was initiated a study of properties of Leibniz algebras relative to the Liezation functor, which assigns to a Leibniz algebra $\mathfrak{q}$ the Lie algebra $\mathfrak{q}_{\Lie} = \mathfrak{q}/\langle \{[x,x]: x \in \mathfrak{q} \} \rangle$, as opposed to the absolute ones, the corresponding to the abelianization functor. The origin of this point of view comes from the general theory of central extensions relative to a chosen subcategory  of a base category introduced in \cite{RefJK} and considered in the context of semi-abelian categories relative to a Birkhoff subcategory in \cite{RefJMT}.

Continuing with this study, in the first section, we introduce the concept of \textit{{\Lie}-isoclinism} for pairs of Leibniz algebras 
that is an equivalence relation. Similar to the pair of Lie algebras, \cite{RefMP}, we prove that {\Lie}-isoclinic family of Leibniz algebras contains at least one {\Lie}-stem Leibniz algebra which is the smallest dimension and give some results about this concept, as well.

In section $ 3 $, we use a function, named factor set, which is introduced by non-abelian extension of Leibniz algebras. Note that, this function (without indicating on factor set) has given by Liu et al. \cite{RefLSW} to classify non-abelian extensions of Leibniz algebras by the second non-abelian cohomology of Leibniz algebras. 

Finally, in section $4$, we show that two pairs of the same finite dimensional {\Lie}-isoclinic ({\Lie}-stem) Leibniz algebras are not isomorphic and indicate some relevant counterexamples. Moreover, by using the concept of factor set, we present as our main result some conditions that {\Lie}-isoclinism and isomorphism, for finite dimensional {\Lie}-stem Leibniz algebras, are equal.

Throughout, all Leibniz algebras are considered over a fixed field $ \mathbb{K} $, unless otherwise stated.
Our basic assumptions are the following.

\begin{definition}
Let $ \mathfrak{m} $ be  a two-sided ideal of the Leibniz algebra $ \mathfrak{q} $, then $ (\mathfrak{m}, \mathfrak{q}) $ is said to be a \textit{pair} of Leibniz algebras.
\end{definition}
\begin{definition}
The \textit{{\Lie}-commutator} and \textit{{\Lie}-center} of the pair of $ (\mathfrak{m}, \mathfrak{q}) $
are both two-sided ideals of $\mathfrak{q}$ contained in $\mathfrak{m}$
\begin{center}
$[\mathfrak{m}, \mathfrak{q}]_{\Lie}=\langle\{ [m, q]+[q, m]\mid m\in\mathfrak{m}, q\in\mathfrak{q} \}\rangle $\\
$Z_{\Lie}(\mathfrak{m}, \mathfrak{q})=\lbrace m\in \mathfrak{m} \mid [m,q]+[q, m]=0~{\rm for~all}~q\in \mathfrak{q}\rbrace  = Z_{\Lie}(\mathfrak{q}) \cap \mathfrak{m}$.
\end{center}
\end{definition}
\begin{remark}
When $\mathfrak{m}=\mathfrak{q}$, then $Z_{\Lie}(\mathfrak{q}, \mathfrak{q})$ coincides with the {\Lie}-center of $\mathfrak{q}$ given in \cite{RefCKh}.
\end{remark}

\section{{\Lie}-isoclinism of  pairs of Leibniz algebras}
\label{sec:1}
 We begin with the following definition which is the corresponding relative version of the isoclinism of Lie algebras given in \cite{RefMP} (absolute case for Lie algebras).

\begin{definition} \label{isoclinism1}
The pairs of Leibniz algebras  $ (\mathfrak{m}_i, \mathfrak{q}_i), i=1,2$, are said to be {\Lie}-isoclinic if there exist
isomorphisms $\alpha:\mathfrak{q_1}/Z_{\Lie}(\mathfrak{m_1}, \mathfrak{q_1}) \longrightarrow \mathfrak{q_2}/Z_{\Lie}(\mathfrak{m_2}, \mathfrak{q_2})$ with $ \alpha (\mathfrak{m_1}/Z_{\Lie}(\mathfrak{m_1}, \mathfrak{q_1}))=  \mathfrak{m_2}/Z_{\Lie}(\mathfrak{m_2}, \mathfrak{q_2}) $ and $ \beta: [\mathfrak{m_1}, \mathfrak{q_1}]_{\Lie}\longrightarrow[\mathfrak{m_2}, \mathfrak{q_2}]_{\Lie} $ such that the following diagram is commutative:
\begin{equation} \label{isoclinism}
\begin{CD}
\frac{\mathfrak{m_1}}{Z_{\Lie}(\mathfrak{m_1}, \mathfrak{q_1})}\times \frac{\mathfrak{q_1}}{Z_{\Lie}(\mathfrak{m_1}, \mathfrak{q_1})}@>{C_1}>>[\mathfrak{m_1}, \mathfrak{q_1}]_{\Lie}\\
@VV{{\alpha_{|}}\times \alpha}V @VV{\beta}V\\
\frac{\mathfrak{m_2}}{Z_{\Lie}(\mathfrak{m_2}, \mathfrak{q_2})}\times \frac{\mathfrak{q_2}}{Z_{\Lie}(\mathfrak{m_2}, \mathfrak{q_1})}@>{C_2}>>[\mathfrak{m_2}, \mathfrak{q_2}]_{\Lie},
\end{CD}
\end{equation}
where $ C_i(\bar{m_i},\bar{q_i})=[m_i,q_i]+[q_i, m_i] $, for all ${\bar{m_i}}\in \frac{\mathfrak{m}_i}{Z_{\Lie}(\mathfrak{q}_i, \mathfrak{m}_i)}$ and ${\bar{q_i}}\in \frac{\mathfrak{q}_i}{Z_{\Lie}(\mathfrak{q}_i, \mathfrak{m}_i)}, i=1,2$.

In this case, the pair $(\alpha, \beta)$ is called a {\Lie}-isoclinism between $(\mathfrak{m_1}, \mathfrak{q_1})$ and $ (\mathfrak{m_2}, \mathfrak{q_2}) $ and we write $ (\mathfrak{m_1}, \mathfrak{q_1}) \sim (\mathfrak{m_2}, \mathfrak{q_2})$.
\end{definition}
The following Proposition provides an equivalent condition for {\Lie}-isoclinism between two pairs of Leibniz algebras.
\begin{proposition}
Let $ {\pi}_i: \mathfrak{q}_i \twoheadrightarrow \frac{\mathfrak{q}_i}{Z_{\Lie}(\mathfrak{m}_i, \mathfrak{q}_i)}, i= 1, 2$, be the canonical surjective homomorphisms and $\alpha:\mathfrak{q_1}/Z_{\Lie}(\mathfrak{m_1}, \mathfrak{q_1}) \longrightarrow \mathfrak{q_2}/Z_{\Lie}(\mathfrak{m_2}, \mathfrak{q_2})$ with $ \alpha (\mathfrak{m_1}/Z_{\Lie}(\mathfrak{m_1}, \mathfrak{q_1}))= \mathfrak{m_2}/Z_{\Lie}(\mathfrak{m_2}, \mathfrak{q_2}) $ and $ \beta: [\mathfrak{m_1}, \mathfrak{q_1}]_{\Lie}\longrightarrow[\mathfrak{m_2}, \mathfrak{q_2}]_{\Lie} $ be isomorphisms.The pair $(\alpha, \beta)$ is a {\Lie}-isoclinism between $(\mathfrak{m_1}, \mathfrak{q_1})$ and $ (\mathfrak{m_2}, \mathfrak{q_2}) $, if and only if $\beta([m_1,q_1]+[q_1, m_1])=[m_2,q_2]+[q_2, m_2]$, where $ m_2\in  \mathfrak{m_2}, q_2\in  \mathfrak{q_2} $, $\alpha(\pi_1(q_1))=\pi_2(q_2)$ and $\alpha(\pi_1(m_1))=\pi_2(m_2)$.
\end{proposition}
\begin{proof}
Direct checking.
\end{proof}

\begin{remark}
When $\mathfrak{m}_i = \mathfrak{q}_i, i = 1, 2$, then we recover the notion of {\Lie}-isoclinism of Leibniz algebras given in \cite{RefBC}.
\end{remark}
An immediate result from Definition \ref{isoclinism1} is the following
\begin{corollary}\label{corollary1}
Let  the pairs of Leibniz algebras $( \mathfrak{m}, \mathfrak{q})$ and  $( \mathfrak{n}, \mathfrak{p})$ be  {\Lie}-isoclinic. Then $\mathfrak{m}$ and $\mathfrak{n}$ are  {\Lie}-isoclinic.
\end{corollary}

The following Lemma yields information about the {\Lie}-isoclinism between a pair of Leibniz algebras and its quotient pair by a two-sided ideal.

\begin{lemma}\label{lemma1}
Let $ (\mathfrak{m}, \mathfrak{q})$ be a pair of Leibniz algebras and $ \mathfrak{n} $ a two-sided ideal of $ \mathfrak{q} $ contained in $ \mathfrak{m} $. Then $ (\frac{\mathfrak{m}}{\mathfrak{n}}, \frac{\mathfrak{q}}{\mathfrak{n}}) \sim (\frac{\mathfrak{m}}{\mathfrak{n \cap [\mathfrak{m}, \mathfrak{q}]_{\Lie}}}, \frac{\mathfrak{q}}{\mathfrak{n \cap  [\mathfrak{m}, \mathfrak{q}]_{\Lie}}})$. In particular, $\mathfrak{n} \cap [\mathfrak{m}, \mathfrak{q}]_{\Lie}=0 $ if and only if $ (\frac{\mathfrak{m}}{\mathfrak{n}}, \frac{\mathfrak{q}}{\mathfrak{n}}) \sim (\mathfrak{m}, \mathfrak{q}) $.
\end{lemma}
\begin{proof}
We set $ \overline{\mathfrak{q}}=\frac{\mathfrak{q}}{\mathfrak{n}},  \overline{\mathfrak{m}}=\frac{\mathfrak{m}}{\mathfrak{n}}, \widetilde{\mathfrak{q}}=\frac{\mathfrak{q}}{\mathfrak{n \cap  [\mathfrak{m}, \mathfrak{q}]_{\Lie}}}$ and $\widetilde{\mathfrak{m}}=\frac{\mathfrak{m}}{\mathfrak{n \cap [\mathfrak{m}, \mathfrak{q}]_{\Lie}}}$.

 It is easy to check that the map $\gamma : \widetilde{\mathfrak{q}} \to  \overline{\mathfrak{q}}$ given by $\gamma(\widetilde{q}) = \gamma(q + (\mathfrak{n} \cap [\mathfrak{m}, \mathfrak{q}]_{\Lie})) = \overline{q} = q +  \mathfrak{n}$, is a surjective homomorphism such that $\gamma\left( Z_{\Lie}( \widetilde{\mathfrak{m}},  \widetilde{\mathfrak{q}}) \right) = Z_{\Lie}(\overline{\mathfrak{m}}, \overline{\mathfrak{q}})$, then it induces a surjective homomorphism
 $\alpha: \dfrac{ \widetilde{\mathfrak{q}}}{Z_{\Lie}( \widetilde{\mathfrak{m}},  \widetilde{\mathfrak{q}})}\twoheadrightarrow \dfrac{\overline{\mathfrak{q}}}{Z_{\Lie}(\overline{\mathfrak{m}}, \overline{\mathfrak{q}})}$, given by $\alpha ( \widetilde{q}+Z_{\Lie}( \widetilde{\mathfrak{m}},  \widetilde{\mathfrak{q}}))=\overline{q}+Z_{\Lie}(\overline{\mathfrak{m}}, \overline{\mathfrak{q}})$. Moreover $\alpha$ is injective, because $\alpha(\widetilde{q}+Z_{\Lie}( \widetilde{\mathfrak{m}},  \widetilde{\mathfrak{q}}))=0$, implies that $\overline{q} \in Z_{\Lie}(\overline{\mathfrak{m}}, \overline{\mathfrak{q}})$, that is $\overline{q} = \gamma(\widetilde{q})$ with $\widetilde{q} \in {Z_{\Lie}( \widetilde{\mathfrak{m}},  \widetilde{\mathfrak{q}})}$. Consequently, $\alpha$ is an isomorphism.

 On the other hand, the restriction of $\gamma$ provides the surjective homomorphism $\beta:[ \widetilde{\mathfrak{m}},  \widetilde{\mathfrak{q}}]_{\Lie} \twoheadrightarrow [\overline{\mathfrak{m}}, \overline{\mathfrak{q}}]_{\Lie},$ given by $\beta([ \widetilde{m}, \widetilde{q}]+[\widetilde{q}, \widetilde{m}])=[\overline{m},\overline{q}]+[\overline{q},\overline{m}]$. Moreover it is easy to check that $\beta$ is injective.

 Now the commutativity of diagram (\ref{isoclinism}) is obvious.

For the second statement,  if $\mathfrak{n} \cap [\mathfrak{m}, \mathfrak{q}]_{\Lie}=0 $ then $ \widetilde{\mathfrak{m}}\cong \mathfrak{m} $ and $\tilde{\mathfrak{q}}\cong \mathfrak{q}$ and so $(\frac{\mathfrak{m}}{\mathfrak{n}}, \frac{\mathfrak{q}}{\mathfrak{n}}) \sim (\mathfrak{m}, \mathfrak{q}) $.
Conversely, the isomorphism $\beta : [\mathfrak{m}, \mathfrak{q}]_{\Lie} \to [\overline{\mathfrak{m}}, \overline{\mathfrak{q}}]_{\Lie}$ actually is induced by the canonical projection $\pi : \mathfrak{q} \twoheadrightarrow \frac{\mathfrak{q}}{\mathfrak{n}}$ and $\Ker(\beta) = \mathfrak{n} \cap [\mathfrak{m}, \mathfrak{q}]_{\Lie}$
\end{proof}

\begin{definition}
The pair of Leibniz algebras $(\mathfrak{m}, \mathfrak{q})$ is said to be a \textit{{\Lie}-stem pair} of Leibniz algebras, when
$Z_{\Lie}(\mathfrak{m}, \mathfrak{q})\subseteq [\mathfrak{m}, \mathfrak{q}]_{\Lie} $.
\end{definition}

\begin{example}
An example of {\Lie}-stem pair is provided by the three-dimensional Leibniz algebra $\mathfrak{q}$ with basis $\{a_1,a_2,a_3\}$ and bracket operation given by $[a_1,a_3]=a_1, [a_2,a_3]=a_2$ (class 3 {\it a)} in \cite{RefCIL}), and the two-sided ideal $\mathfrak{m} = span \{a_1\}$. Obviously $0 =Z_{\Lie}(\mathfrak{m},\mathfrak{q})\subseteq   [\mathfrak{m},\mathfrak{q}]_{\Lie} = span \{a_1\}$.
\end{example}

\begin{proposition}\label{proposition2}
The pair of Leibniz algebras $(\mathfrak{m}, \mathfrak{q})$ is a {\Lie}-stem pair  if and only if the unique two-sided ideal $\mathfrak{s}$ of  $\mathfrak{q}$, such that $\mathfrak{s} \subseteq \mathfrak{m}$ and $\mathfrak{s}\cap [\mathfrak{m}, \mathfrak{q}]_{\Lie} =0$, is the trivial one.
\end{proposition}

\begin{proof}
Let $ \mathfrak{s} $ be a two-sided ideal of $ \mathfrak{q}$ such that $\mathfrak{s}\subseteq \mathfrak{m}$ and $ \mathfrak{s}\cap [\mathfrak{m}, \mathfrak{q}]_{\Lie} =0$, then $ \mathfrak{s}\subseteq Z_{\Lie}(\mathfrak{m}, \mathfrak{q})$ and so $ \mathfrak{s}=\mathfrak{s} \cap Z_{\Lie}(\mathfrak{m}, \mathfrak{q})\subseteq \mathfrak{s} \cap [\mathfrak{m}, \mathfrak{q}]_{\Lie}=0$.\\
Conversely, assume that on the contrary $Z_{\Lie}(\mathfrak{m}, \mathfrak{q})\nsubseteq [\mathfrak{m}, \mathfrak{q}]_{\Lie}$, then there exists $z\in Z_{\Lie}(\mathfrak{m}, \mathfrak{q})\backslash [\mathfrak{m}, \mathfrak{q}]_{\Lie}, z \neq 0$.
Put $ \mathfrak{s}$ the two sided-ideal of $\mathfrak{q}$ spanned by $z$, which is contained in $\mathfrak{m}$. So from the assumption we have $ \mathfrak{s}\cap [\mathfrak{m}, \mathfrak{q}]_{\Lie} =0 $. Hence $ \mathfrak{s}=0$ and a contradiction follows. 
\end{proof}

\vspace*{-.23cm}
Plainly, {\Lie}-isoclinism between the pairs of Leibniz algebras is an equivalence relation, (see for instance \cite{RefBC}). By the {\Lie}-isoclinic family we mean equivalent classes that contains the class of all pairs of Leibniz algebras. The following Theorem ensures the existence of a {\Lie}-stem pair in the {\Lie}-isoclinic family of pairs of Leibniz algebras.
\begin{theorem}\label{theorem1}
Every {\Lie}-isoclinic family $\mathcal{C}$ of pairs of Leibniz algebras contains at least one {\Lie}-stem pair of Leibniz algebras.
\end{theorem}

\begin{proof}
Let $ (\mathfrak{m}, \mathfrak{q}) $ be an arbitrary pair of Leibniz algebras in $ \mathcal{C} $ and $  \mathcal{A}=\lbrace\mathfrak{s} \mid \mathfrak{s}\trianglelefteq \mathfrak{q}, \mathfrak{s}\subseteq \mathfrak{m}, \mathfrak{s}\cap [\mathfrak{m}, \mathfrak{q}]_{\Lie} =0 \rbrace $. The set $ \mathcal{A}$ is non-empty because it contains at least the zero ideal. We define a partial ordering on $\mathcal{A}$ by inclusion and
evidently, by Zorn's lemma, we can find a maximal two-sided ideal $  \mathfrak{s}$ in $\mathcal{A}$. Since $ \mathfrak{s}\cap [\mathfrak{m}, \mathfrak{q}]_{\Lie} =0 $, it follows from Lemma \ref{lemma1} that $(\mathfrak{m/s}, \mathfrak{q/s})\in \mathcal{C}$. Now, suppose that $ \mathfrak{h/s} $ is a two-sided ideal of $ \mathfrak{q/s}$
contained in $ \mathfrak{m/s} $ such that $ \mathfrak{h/s}\cap [\mathfrak{m/s}, \mathfrak{q/s}]_{\Lie} =0$.  Note that such a $\mathfrak{h/s}$ always exists; for instance $\mathfrak{h}= \mathfrak{s}$. Then we have
$ \mathfrak{h}\cap [\mathfrak{m}, \mathfrak{q}]_{\Lie}\subseteq \mathfrak{s}\cap [\mathfrak{m}, \mathfrak{q}]_{\Lie} =0$ and so $ \mathfrak{h}\in \mathcal{A} $. Moreover, $ \mathfrak{s}\subseteq \mathfrak{h} $, so by the maximality of $ \mathfrak{s} $, it follows that $\mathfrak{h}=\mathfrak{s}  $ and then $ \mathfrak{h/s}=0 $. Therefore, by virtue of Proposition \ref{proposition2}, $(\mathfrak{m/s}, \mathfrak{q/s}) $ is a {\Lie}-stem pair of Leibniz algebras, as required.
\end{proof}

One of the main results in this paper is the following
\begin{theorem}\label{theorem2}
Let  $\mathcal{C} $ be a {\Lie}-isoclinic family of finite dimensional pairs of Leibniz algebras and $ (\mathfrak{n}, \mathfrak{p})\in \mathcal{C}$. Then $ (\mathfrak{n}, \mathfrak{p})$ is a {\Lie}-stem pair if and only if
${\rm dim}(\mathfrak{p})= {\rm min} \lbrace {\rm dim}(\mathfrak{q}) | (\mathfrak{m}, \mathfrak{q})\in \mathcal{C}\rbrace$.
\end{theorem}

\begin{proof}
Let $ (\mathfrak{m}, \mathfrak{q})$  and $(\mathfrak{n}, \mathfrak{p})$ be arbitrary pairs in $\mathcal{C}$ and assume that $(\mathfrak{n}, \mathfrak{p})$  is a {\Lie}-stem pair such that $\mathfrak{p}$ is finite-dimensional. Then we have
\begin{align*}
\frac{[\mathfrak{m}, \mathfrak{q}]_{\Lie}}{[\mathfrak{m}, \mathfrak{q}]_{\Lie}\cap Z_{\Lie}(\mathfrak{m}, \mathfrak{q})}&\cong \frac{[\mathfrak{m}, \mathfrak{q}]_{\Lie}+Z_{\Lie}(\mathfrak{m}, \mathfrak{q})}{Z_{\Lie}(\mathfrak{m}, \mathfrak{q})}\\
&\cong \left[\frac{\mathfrak{m}}{Z_{\Lie}(\mathfrak{m}, \mathfrak{q})},\frac{\mathfrak{q}}{Z_{\Lie}(\mathfrak{m}, \mathfrak{q})}\right]_{\Lie}\\
&\cong \left[\frac{\mathfrak{n}}{Z_{\Lie}(\mathfrak{n}, \mathfrak{p})},\frac{\mathfrak{p}}{Z_{\Lie}(\mathfrak{n}, \mathfrak{p})} \right]_{\Lie}\\
&\cong \frac{[\mathfrak{n}, \mathfrak{p}]_{\Lie}}{Z_{\Lie}(\mathfrak{n}, \mathfrak{p})},
\end{align*}
and $[\mathfrak{m}, \mathfrak{q}]_{\Lie}\cong [\mathfrak{n}, \mathfrak{p}]_{\Lie}$. Therefore ${\rm dim}\left(Z_{\Lie}(\mathfrak{n}, \mathfrak{p})\right)= {\rm dim} \left([\mathfrak{m}, \mathfrak{q}]_{\Lie}\cap Z_{\Lie}(\mathfrak{m}, \mathfrak{q})\right)\leq$ ${\rm dim} \left( Z_{\Lie}(\mathfrak{m}, \mathfrak{q})\right)$.
On the other hand $\frac{\mathfrak{p}}{Z_{\Lie}(\mathfrak{n}, \mathfrak{p})} \cong \frac{\mathfrak{q}}{Z_{\Lie}(\mathfrak{m}, \mathfrak{q})}$. Therefore ${\rm dim} \left( \mathfrak{p} \right) \leq  {\rm dim} \left( \mathfrak{q}\right)$.

Conversely, let $(\mathfrak{n}, \mathfrak{p})$ be in the family  $ \mathcal{C} $ such that $\mathfrak{p} $ has the minimum dimension. Owing to Theorem \ref{theorem1} there is a two-sided ideal $\mathfrak{t}$ of $\mathfrak{p}$ contained in $Z_{\Lie}(\mathfrak{n}, \mathfrak{p})$ such that $(\mathfrak{n}, \mathfrak{p})\sim (\frac{\mathfrak{n}}{\mathfrak{t}}, \frac{\mathfrak{p}}{\mathfrak{t}})$ and $Z _{\Lie}(\mathfrak{n}, \mathfrak{p})= \left([\mathfrak{n}, \mathfrak{p}]_{\Lie} \cap Z_{\Lie}(\mathfrak{n}, \mathfrak{p}) \right)\oplus  \mathfrak{t}$. But $\mathfrak{p}$ has minimum dimension, which implies that $\mathfrak{t}=0$, therefore $Z_{\Lie}(\mathfrak{n}, \mathfrak{p})\subseteq [\mathfrak{n}, \mathfrak{p}]_{\Lie}$ and this completes the proof.
\end{proof}

The above Theorem provides the following interesting consequence which will be used in the sequel and our main result in the last section.

\begin{corollary} \label{corollary2}
If $(\mathfrak{m}, \mathfrak{q})$ and $(\mathfrak{n}, \mathfrak{p})$ are two {\Lie}-isoclinic {\Lie}-stem pairs of Leibniz algebras then $Z_{\Lie}(\mathfrak{m}, \mathfrak{q})\cong Z_{\Lie}(\mathfrak{n}, \mathfrak{p})$.
\end{corollary}

\begin{proof}
Let $(\mathfrak{m}, \mathfrak{q})$ and $(\mathfrak{n}, \mathfrak{p})$ be two {\Lie}-isoclinic pairs of Leibniz algebras. In view of proof of Theorem \ref{theorem2}  and isomorphism $\beta: [\mathfrak{m}, \mathfrak{q}]_{\Lie}\longrightarrow[\mathfrak{n}, \mathfrak{p}]_{\Lie}$ we have the following commutative diagram with exact rows:
$$\begin{CD}
0@>>> Z_{\Lie}(\mathfrak{m}, \mathfrak{q})@>>> [\mathfrak{m}, \mathfrak{q}]_{\Lie}@>{\pi_1}>>
\frac{ [\mathfrak{m}, \mathfrak{q}]_{\Lie}}{Z_{\Lie}(\mathfrak{m}, \mathfrak{q})}@>>>0\\
&& @VV{\beta_{|}}V @V{\wr}V{\beta}V @V{\wr}V{\alpha}V\\
~~ 0@>>> Z_{\Lie}(\mathfrak{n}, \mathfrak{p})@>>>[\mathfrak{n}, \mathfrak{p}]_{\Lie}@>{\pi_2}>>
\frac{[\mathfrak{n}, \mathfrak{p}]_{\Lie}}{Z_{\Lie}(\mathfrak{n}, \mathfrak{p})} @>>>0.
\end{CD}$$\\
where $\beta \left( Z_{\Lie}(\mathfrak{m}, \mathfrak{q}) \right) \subseteq Z_{\Lie}(\mathfrak{n}, \mathfrak{p})$ since for all $x\in  [\mathfrak{m}, \mathfrak{q}]_{\Lie}$, $\alpha \left(x+Z_{\Lie}(\mathfrak{m}, \mathfrak{q})\right) = \beta (x) +  Z_{\Lie}(\mathfrak{n}, \mathfrak{p})$. Hence, for $x \in Z_{\Lie}(\mathfrak{m}, \mathfrak{q})$, $0 = \alpha  \left( \pi_1(x) \right) = \beta(x) + Z_{\Lie}(\mathfrak{n}, \mathfrak{p})$, so $\beta(x) \in Z_{\Lie}(\mathfrak{n}, \mathfrak{p})$.
Now the Snake Lemma \cite{RefBB} yields  $\beta_{|}$ is a surjective homomorphism and so $\beta ( Z_{\Lie}(\mathfrak{m}, \mathfrak{q}))= Z_{\Lie}(\mathfrak{n}, \mathfrak{p})$. Moreover the left hand square is a pull-back diagram, then  $\beta_{\mid}$ is a monomorphism. Therefore  $Z_{\Lie}(\mathfrak{m}, \mathfrak{q})\cong Z_{\Lie}(\mathfrak{n}, \mathfrak{p})$.
\end{proof}

\section{Factor sets of a pair of Leibniz algebras}
\label{sec:2}
 Chevalley and Eilenberg in \cite{RefCE} defined the factor sets for Lie algebras and now we recall factor sets for Leibniz algebras from \cite{RefLSW} and we analyze the interplay with  the concepts relative to the Liezation functor for a pair of Leibniz algebras.

Let $0\longrightarrow \mathfrak{m}\xrightarrow {\subseteq}\mathfrak{q}\xrightarrow{\pi}\mathfrak{q^*} \longrightarrow 0$ be a non-abelian extension of Leibniz algebras \cite[Definition 2.5]{RefLSW}. We  choose  a  splitting $\tau: \mathfrak{q^*}\longrightarrow \mathfrak{q}$, that is a linear map  such that  $ \pi \circ \tau = {\sf Id}_{\mathfrak{q^*}}$.  For each  $x\in \mathfrak{q^*}$  we have two linear maps $L_x, R_x: \mathfrak{m}\longrightarrow \mathfrak{m}$ given by $L_x(m)=  [m,  \tau(x)]$ and $R_x(m)=[\tau(x), m]$.
Associated to any  pair of elements $ x, y\in \mathfrak{q^*} $, there is an element $f(x, y)\in \mathfrak{m}$  such that
$$f(x, y) = [\tau(x), \tau(y)]-\tau [x, y].$$

The  linear map $f$  is called  the  {\it factor set} corresponding to the function  $\tau$.

For any $ x, y, z\in \mathfrak{q^*}$, the following identities concerning factor sets hold:
\begin{align*}
[[\tau(x), \tau(y)], \tau(z)]&=[f(x, y), \tau(z)]+[\tau([x,y]), \tau(z)]\\
&=L_z(f(x,y))+f([x,y], z)+\tau([[x, y], z]),
\end{align*}
\begin{align*}
[[\tau(x), \tau(z)], \tau(y)]&=[f(x, z), \tau(y)]+[\tau([x,z]), \tau(y)]\\
&=L_y(f(x,z))+f([x,z], y)+\tau([[x, z], y]),
\end{align*}
\begin{align*}
[\tau(x), [\tau(y), \tau(z)]]&=[\tau(x), f(y, z)]+[\tau(x), \tau([y, z])]\\
&=R_x(f(y, z))+f(x, [y, z])+\tau([x, [y, z]]).
\end{align*}
From the above identities, the following equation is immediately derived:
 \begin{equation} \label{factor set id}
f([x,y], z)-f([x,z], y)-f(x, [y, z])+ L_z(f(x,y))- L_y(f(x,z))-R_x(f(y, z))=0
\end{equation}
Given a splitting $\tau$, we  define on the $\mathbb{K}$-vector space $\mathfrak{m} \oplus \mathfrak{q^*}$  the bracket operation
\begin{equation*}
 [(m_1, x_1), (m_2, x_2)] = ([m_1, m_2]+R_{x_1}(m_2)+L_{x_2}(m_1)+f(x_1, x_2), [x_1, x_2]).
\end{equation*}
A routine computation having in mind equation (\ref{factor set id}) shows that $(\mathfrak{m} \oplus \mathfrak{q^*}, [-,-])$ is a Leibniz algebra, which will be denoted $ \mathfrak{m}\times_f \mathfrak{q^*}$.

Note that $\tau$ is a homomorphism if and only if $f(x, y) = 0$, for all $x, y\in \mathfrak{q^*},$ that is $f$ measures the deficiency of $\tau$ to be a homomorphism.
If the exact sequence $0\longrightarrow \mathfrak{m}\xrightarrow {i}\mathfrak{q}\xrightarrow{\psi}\mathfrak{q^*} \longrightarrow 0$ splits by a homomorphism  $\tau : \mathfrak{q^*} \to \mathfrak{q}$, then $\mathfrak{m}$ is endowed with an action from $\mathfrak{q^*}$ given by $[x,m]=[\tau(x), i(m)]_{\mathfrak{q}}, [m,x]=[i(m), \tau(x)]_{\mathfrak{q}}, m \in \mathfrak{m}, x \in \mathfrak{q^*}$. Hence the semi-direct product $\mathfrak{m} \rtimes \mathfrak{q^*}$ can be constructed, which gives rise to the split extension $0\longrightarrow \mathfrak{m}\xrightarrow{} \mathfrak{m} \rtimes \mathfrak{q^*} \xrightarrow{pr}\mathfrak{q^*} \longrightarrow 0$ \cite{RefLP}.

\begin{definition}
We say that two pairs of Leibniz algebras $(\mathfrak{m}, \mathfrak{q})$ and $(\mathfrak{n}, \mathfrak{p})$ are isomorphic if there exist an isomorphism $ \varphi: \mathfrak{q}\rightarrow \mathfrak{p} $ such that $ \varphi_{|{\mathfrak{m}}}:\mathfrak{m}\cong \mathfrak{n} $.
\end{definition}

\begin{lemma}\label{lemma2}
Let $(\mathfrak{m}, \mathfrak{q})$ be a pair of Leibniz algebras. Then there exists the factor set $f:\dfrac{\mathfrak{q}}{\mathfrak{m}}\times \dfrac{\mathfrak{q}}{\mathfrak{m}}\longrightarrow \mathfrak{m}$ such that $\left( \Ker(\pi), \mathfrak{m}\times_f \dfrac{\mathfrak{q}}{\mathfrak{m}} \right) \cong \left( \mathfrak{m}, \mathfrak{q} \right)$, where $\pi : \mathfrak{m}\times_f \dfrac{\mathfrak{q}}{\mathfrak{m}} \twoheadrightarrow \dfrac{\mathfrak{q}}{\mathfrak{m}}$   is the canonical projection.
\end{lemma}
\begin{proof}
Let $ \mathfrak{h} $ be a vector space complement of $ \mathfrak{m} $ in $ \mathfrak{q} $ and $ \rho :\dfrac{\mathfrak{q}}{\mathfrak{m}}\longrightarrow \mathfrak{q} $ be a linear map given by $ \rho(\bar{x})=h$, where $ x=h+m $ with $h\in \mathfrak{h}$ and  $m\in \mathfrak{m}$,  which is a splitting of the extension $0 \to \mathfrak{m} \to \mathfrak{q} \overset{pr} \to \dfrac{\mathfrak{q}}{\mathfrak{m}} \to 0$.

Now, we define $ f:\dfrac{\mathfrak{q}}{\mathfrak{m}}\times \dfrac{\mathfrak{q}}{\mathfrak{m}}\longrightarrow \mathfrak{m} $ by $ f(\bar{x}, \bar{y})=[\rho(\bar{x}), \rho(\bar{y})] -\rho([\bar{x}, \bar{y}])$  that is well-defined, since $pr([\rho(\bar{x}), \rho(\bar{y})])=[\bar{x}, \bar{y}]= pr (\rho([\bar{x}, \bar{y}])) $. Hence $ f $ is the factor set and $\mathfrak{m}\times_f \dfrac{\mathfrak{q}}{\mathfrak{m}} $ is  isomorphic to $\mathfrak{q}$ via $(m, \bar{x})\longmapsto m+\rho (\bar{x})$.  Also, it is easy to check that  $f|_{\Ker(\pi)}: \Ker(\pi) \cong  \mathfrak{m}$ and this  completes the proof.
\end{proof}

For the pair of Leibniz algebras $ ( \mathfrak{m}, \mathfrak{q}) $ we consider the extension  \[0\longrightarrow {Z_{\Lie}(\mathfrak{m}, \mathfrak{q})}\xrightarrow {\subseteq}\mathfrak{m}\xrightarrow{\pi}{\mathfrak{m}}/{Z_{\Lie}(\mathfrak{m}, \mathfrak{q})} \longrightarrow 0 \]  and the factor set  $ f $ corresponding to the splitting $ \tau:{\mathfrak{m}}/{Z_{\Lie}(\mathfrak{m}, \mathfrak{q})}\longrightarrow \mathfrak{m}  $ given by Lemma \ref{lemma2}. Moreover $ {Z_{\Lie}(\mathfrak{m}, \mathfrak{q})}\times_f {\mathfrak{m}}/{Z_{\Lie}(\mathfrak{m}, \mathfrak{q})} $ is a Leibniz algebra isomorphic to $\mathfrak{m}$. We henceforth assume that ${\mathfrak{m}}_f:=Z_{\Lie}(\mathfrak{m}, \mathfrak{q}) \times_f {\mathfrak{m}}/{Z_{\Lie}(\mathfrak{m}, \mathfrak{q})}$ is given as just described.
It is easy to see that the following map is an isomorphism:
\begin{equation*}
\kappa: Z_{\Lie}(\mathfrak{m}, \mathfrak{q})\longrightarrow Z_{\Lie}^{\mathfrak{m}_f}=\lbrace (z, 0) \mid z\in Z_{\Lie}(\mathfrak{m}, \mathfrak{q})\rbrace, ~~~~ z\longmapsto (z, 0).
\end{equation*}

This notation and the previous Lemma give rise to the following
\begin{proposition} \label{proposition3}
 Let $  (\mathfrak{m}, \mathfrak{q}) $ and $  (\mathfrak{n}, \mathfrak{p}) $ be two {\Lie}-isoclinic {\Lie}-stem pairs of Leibniz algebras. Then there exists a factor set $ f :{\mathfrak{m}}/{{Z_{\Lie}(\mathfrak{m}, \mathfrak{q})}}\times  {\mathfrak{m}}/{{Z_{\Lie}(\mathfrak{m}, \mathfrak{q})}}\to Z_{\Lie}(\mathfrak{m}, \mathfrak{q})$ such that $$ \mathfrak{n}\cong Z_{\Lie}(\mathfrak{m}, \mathfrak{q})\times_f  \dfrac{\mathfrak{m}}{{Z_{\Lie}(\mathfrak{m}, \mathfrak{q})}}={\mathfrak{m}}_f$$
 \end{proposition}
 \begin{proof}
 Let the pair $(\alpha, \beta) $ be the {\Lie}-isoclinism between $(\mathfrak{m}, \mathfrak{q})$ and $ (\mathfrak{n}, \mathfrak{p}) $. By Corollary \ref{corollary2} we have  the isomorphism $ \beta : {Z_{\Lie}(\mathfrak{m}, \mathfrak{q})} \cong {Z_{\Lie}(\mathfrak{n}, \mathfrak{p})} $ and owing to Lemma \ref{lemma2} there exists a factor set $g:{\mathfrak{n}}/{Z_{\Lie}(\mathfrak{n}, \mathfrak{p})}\times {\mathfrak{n}}/{Z_{\Lie}(\mathfrak{n}, \mathfrak{p})}$ $\longrightarrow {Z_{\Lie}(\mathfrak{n}, \mathfrak{p})}$ such that $ {\mathfrak{n}} \cong {\mathfrak{n}}_g$.   Hence we define the map $ f : {\mathfrak{m}}/{{Z_{\Lie}(\mathfrak{m}, \mathfrak{q})}}\times {\mathfrak{m}}/{{Z_{\Lie}(\mathfrak{m}, \mathfrak{q})}}\longrightarrow Z_{\Lie}(\mathfrak{m}, \mathfrak{q})$ given by $ f(\bar{m_1}, \bar {m_2})=\beta^{-1}\big(g( \alpha \times \alpha(\bar{m_1}, \bar{m_2}))\big)$, for all $ \bar{m_i}\in {\mathfrak{m}}/{{Z_{\Lie}(\mathfrak{m}, \mathfrak{q})}}, i= 1, 2 $, which is $\beta^{-1}$ of the factor set corresponding to $\rho_2 \circ \alpha$, where $\rho_2$ is the splitting of $\pi_2 : \mathfrak{n} \twoheadrightarrow{\mathfrak{n}}/{Z_{\Lie}(\mathfrak{n}, \mathfrak{p})}$.  One readily see that the mapping
\begin{equation*}
 \theta : Z_{\Lie}(\mathfrak{m}, \mathfrak{q})\times_f  \frac{\mathfrak{m}}{{Z_{\Lie}(\mathfrak{m}, \mathfrak{q})}}\longrightarrow Z_{\Lie}(\mathfrak{n}, \mathfrak{p})\times_g  \frac{\mathfrak{n}}{{Z_{\Lie}(\mathfrak{n}, \mathfrak{p})}}
\end{equation*}
defined by $ \theta (z_1, \bar{m_1})=(\beta (z_1), \alpha (\bar{m_1})),  z_1 \in Z_{\Lie}(\mathfrak{m}, \mathfrak{q}), \bar{m_1}  \in \frac{\mathfrak{m}}{{Z_{\Lie}(\mathfrak{m}, \mathfrak{q})}}$, is an isomorphism, as required.
 \end{proof}
We close this section by the following Lemma and Proposition which are of interest in their own account.
\begin{lemma} \label{lemma3}
 Let $ f $ and $ g $ be two factor sets on the pair of Leibniz algebras $ (\mathfrak{m}, \mathfrak{q}) $ and $ (\mathfrak{n}, \mathfrak{p}) $, respectively. If $ \eta: \mathfrak{m}_f\longrightarrow \mathfrak{n}_g$ is an isomorphism such that $ \eta(Z_{\Lie}^{\mathfrak{m}_f})=Z_{\Lie}^{\mathfrak{n}_g}$, then $ \eta $ induces isomorphisms $ \eta_{1}: \mathfrak{m}/ Z_{\Lie}(\mathfrak{m}, \mathfrak{q})\longrightarrow \mathfrak{n}/ Z_{\Lie}(\mathfrak{n}, \mathfrak{p}) $ and $ \eta_{2}:Z_{\Lie}(\mathfrak{m}, \mathfrak{q})\longrightarrow Z_{\Lie}(\mathfrak{n}, \mathfrak{p}) $.
\end{lemma}
\begin{proof}
We consider the following commutative diagram:
$$\begin{CD}
0@>>> Z_{\Lie}^{\mathfrak{m}_f}@>>>\mathfrak{m}_f@>{\pi_{1}}>>
\dfrac{\mathfrak{m}_f}{Z_{\Lie}^{\mathfrak{m}_f}}@>>>0\\
&& @VV{\eta_{\mid}}V@VV{\eta}V @VV{\bar{\eta}}V\\
~~ 0@>>> Z_{\Lie}^{\mathfrak{n}_g}@>>>\mathfrak{n}_g@>{\pi_{2}}>>
\dfrac{ \mathfrak{n}_g}{Z_{\Lie}^{\mathfrak{n}_g}} @>>>0,
\end{CD}$$\\
where $\bar{\eta}\left((z, \bar{m})+Z_{\Lie}^{\mathfrak{m}_f}\right)=\eta (z, \bar{m})+Z_{\Lie}^{\mathfrak{n}_g}$.

Now we define $ \eta_1: \mathfrak{m}/ Z_{\Lie}(\mathfrak{m}, \mathfrak{q})\longrightarrow \mathfrak{n}/ Z_{\Lie}(\mathfrak{n}, \mathfrak{p})$ by $ \bar{\eta}((0, \bar{m})+Z_{\Lie}^{\mathfrak{m}_f})=\eta(0, \bar{m})+Z_{\Lie}^{\mathfrak{n}_g}=(0, \eta_1(\bar{m}))+Z_{\Lie}^{\mathfrak{n}_g} $, for all $ \bar{m}\in \mathfrak{m}/ Z_{\Lie}(\mathfrak{m}, \mathfrak{q}) $,
and $\eta_{2}:Z_{\Lie}(\mathfrak{m}, \mathfrak{q})\longrightarrow Z_{\Lie}(\mathfrak{n}, \mathfrak{p})$ given by  $\eta(z,0)=(\eta_2(z), 0)$, for any $z \in  Z_{\Lie}(\mathfrak{m}, \mathfrak{q})$.
It is easy to check that $\eta_{1} $ and $  \eta_{2} $ are isomorphisms.
\end{proof}

\begin{proposition}\label{proposition4}
Let $f$ and $g$ be two factor sets on the pair of Leibniz algebras $(\mathfrak{m}, \mathfrak{q})$ and $(\mathfrak{n}, \mathfrak{p})$, respectively. Let $\eta $, $ \eta_{1}$ and $\eta_{2} $ be as in Lemma \ref{lemma3}. Then there exists a linear map $ d: \mathfrak{n}/ Z_{\Lie}(\mathfrak{n}, \mathfrak{p})\longrightarrow Z_{\Lie}(\mathfrak{n}, \mathfrak{p})$ such that
\begin{equation*} \label{formula}
 \begin{aligned}[]
&\eta_2 {\big([z_1, z_2]+[\mu(\bar{{m}_1}), z_2]+[z_1, \mu(\bar{{m}_2})]+f(\bar{{m}_1}, \bar{{m}_2})\big)}+d(\eta_1([\bar{{m}_1}, \bar{{m}_2}]))=\\
&[\eta_2 (z_1)+d(\eta_1(\bar{{m}_1})), \eta_2 (z_2)+d(\eta_1(\bar{{m}_2}))]+[\eta_2 (z_1)+d(\eta_1(\bar{{m}_1})), \nu(\eta_1{(\bar{{m}_2})}) ]+ \\
&[\nu(\eta_1{(\bar{{m}_1})}), \eta_2 (z_2)+d(\eta_1(\bar{{m}_2}))]+g(\eta_{1}(\bar{m_1}), \eta_{1}(\bar{m_2})),
\end{aligned}
\end{equation*}
 for all  $(z_1, \bar{m_1}), (z_2, \bar{m_2})\in \mathfrak{m}_f$, where $\mu$ and $\nu$ are the corresponding splittings of  $\pi_1: \mathfrak{m} \twoheadrightarrow {\mathfrak{m}}/{Z_{\Lie}(\mathfrak{m}, \mathfrak{q})}$ and $\pi_2: \mathfrak{n} \twoheadrightarrow {\mathfrak{n}}/{Z_{\Lie}(\mathfrak{n}, \mathfrak{p})}$ associated to $f$ and $g$, respectively.
\end{proposition}
\begin{proof}
By Lemma \ref{lemma3}, for all $\bar{m}\in \mathfrak{m}/ Z_{\Lie}(\mathfrak{m}, \mathfrak{q})$ and $z\in  Z_{\Lie}(\mathfrak{m}, \mathfrak{q})$, there exists $z_{\eta_1(\bar{m})}\in  Z_{\Lie}^{\mathfrak{n}_g} $ such that $ \eta (0, \bar{m})=(0, \eta_{1}(\bar{m}))+(z_{\eta_1(\bar{m})},0) $. We define the map $ d:\frac{\mathfrak{n}}{Z_{\Lie}(\mathfrak{n}, \mathfrak{p})}\longrightarrow Z_{\Lie}(\mathfrak{n}, \mathfrak{p})$  by $d(\eta_1(\bar{m}))=z_{\eta_1(\bar{m})}$, which is a linear map. So we have $\eta(z, \bar{m})=\eta(z,0)+\eta(0, \bar{m}) =(\eta_{2}(z), 0) +(0, \eta_{1}(\bar{m})) +(d(\eta_1({\bar{m}})),0)$. Applying this equality, for all $ (z_1, \bar{m_1}), (z_2, \bar{m_2})\in \mathfrak{m}_f $, one gets,
\begin{align*}
\eta([(z_1, \bar{{m}_1}), (z_2, \bar{{m}_2})])=&\eta  ([z_1, z_2]+[\mu(\bar{{m}_1}), z_2]+[z_1, \mu(\bar{{m}_2})]+f(\bar{{m}_1}, \bar{{m}_2}), [\bar{{m}_1}, \bar{{m}_2}])\\
=& \left( \eta_2 {\left([z_1, z_2]+[\mu(\bar{{m}_1}), z_2]+[z_1, \mu(\bar{{m}_2})]+f(\bar{{m}_1}, \bar{{m}_2})\right)} \right.\\
+& \left. d(\eta_1([\bar{{m}_1}, \bar{{m}_2}]), {{\eta}_1}([\bar{{m}_1}, \bar{{m}_2}])\right).
\end{align*}
On the other hand,
\begin{align*}
\eta([(z_1, \bar{{m}_1}), (z_2, \bar{{m}_2})])=& [\eta(z_1, \bar{{m}_1}), \eta(z_2, \bar{{m}_2})]\\
= &[(\eta_2 (z_1)+d(\eta_1(\bar{{m}_1}), \eta_1{(\bar{{m}_1})}), (\eta_2 (z_2)+d(\eta_1(\bar{{m}_2}),  \eta_1{(\bar{{m}_2})})]\\
= & \left( [\eta_2 (z_1)+d(\eta_1(\bar{{m}_1})), \eta_2 (z_2)+d(\eta_1(\bar{{m}_2}))] \right.\\
+ & [\eta_2 (z_1)+d(\eta_1(\bar{{m}_1})), \nu(\eta_1{(\bar{{m}_2})})] \\
+ & [\nu(\eta_1{(\bar{{m}_1})}), \eta_2 (z_2)+d(\eta_1(\bar{{m}_2}))]\\
+ & g(\eta_{1}(\bar{m_1}), \eta_{1}(\bar{m_2})), [\eta_{1}(\bar{m_1}), \eta_{1}(\bar{m_2})]).
\end{align*}
Now the statement follows from the equality of the first component of both computations.
\end{proof}
Under assumptions of above Proposition and Proposition \ref{proposition3} we have the following isomorphism between pairs of Leibniz algebras:
\[(Z_{\Lie}(\mathfrak{m}, \mathfrak{q}), \mathfrak{m}_f) \cong (Z_{\Lie}(\mathfrak{n}, \mathfrak{p}), \mathfrak{n}_g)
\cong (Z_{\Lie}(\mathfrak{n}, \mathfrak{p}), \mathfrak{n}) \]

\section{{\Lie}-isoclinism and isomorphism between pairs of Leibniz algebras}
\label{sec:3}
If two pairs of Leibniz algebras are isomorphic, it is easy to check that they are {\Lie}-isoclinic. But in this chapter, we show that
the converse is not necessarily valid for finite dimensional ({\Lie}-stem) Leibniz algebras, whereas isoclinic and isomorphism are equal for finite dimensional (stem) Lie algebras \cite{RefM}, and Pair of (stem) Lie algebras \cite{RefMP}. Nevertheless, we provide some conditions that these concepts are equal for finite dimensional {\Lie}-stem Leibniz algebras.

\begin{example}
$\mathfrak{q}=span \lbrace a_{1}, a_{2}, a_{3}\rbrace $, with non-zero multiplication $ [a_{1}, a_{3}]= a_{1} $ (it belongs to the class 2 (d)), where $ Z_{\Lie}(\mathfrak{q})=span \lbrace a_{2 }\rbrace $ and $ [\mathfrak{q}, \mathfrak{q}]_{\Lie}=span \lbrace a_{1}\rbrace$,

 and

$\mathfrak{p}= span \lbrace g_{1}, g_{2}, g_{3}\rbrace$, with non-zero multiplications $[g_{1}, g_{3}]=g_{1}, [g_{2}, g_{3}]=g_{2}$ and $[g_{3}, g_{2}]= -g_{2}$  (it belongs to the class 2 (e) with $\alpha=1$), where
$Z_{\Lie}(\mathfrak{p})=span \lbrace g_{2 }\rbrace$ and $[\mathfrak{p}, \mathfrak{p}]_{\Lie}=span \lbrace g_{1}\rbrace$.

Now we define isomorphisms $\omega: \frac{\mathfrak{q}}{Z_{\Lie}(\mathfrak{q})}\longrightarrow \frac{\mathfrak{p}}{Z_{\Lie}(\mathfrak{p})}$ given by $\omega ({a_{i}})={g_{i}},~ i=1, 3$, and $ \tau: [\mathfrak{q}, \mathfrak{q}]_{\Lie}\longrightarrow [\mathfrak{p}, \mathfrak{p}]_{\Lie}$ given by $\tau (a_{1})=g_{1}$.

One easily verifies that $ \mathfrak{p} $ and $ \mathfrak{q}$ are {\Lie}-isoclinic.
\end{example}

In certain circumstances, even with additional conditions like the Leibniz algebras are {\Lie}-stem, this result is not true, as well. In the following, we investigate two {\Lie}-stem (pairs of) Leibniz algebras, with the same finite dimension, that they are not isomorphic since they belong to different classes in \cite[Theorem 4.2.6]{RefD} and \cite[Proposition 3.11]{RefCaKu}, but we check they are {\Lie}-isoclinic.

\begin{example}\ \label{example3}
\begin{enumerate}
\item[{\it a)}] Consider the following five-dimensional non isomorphic {\Lie}-stem Leibniz algebras given in \cite[Theorem 4.2.6]{RefD}:

 $\mathcal{A}_1=span \lbrace a_{1}, a_{2}, a_{3}, a_4, a_5\rbrace $, with non-zero multiplications $[a_{1}, a_{1}]= a_{3}$, $[a_{2}, a_{1}]$ $=a_{4}$ and $[a_{1}, a_{3}]= a_{5}$, in which
  $Z_{\Lie}(\mathcal{A}_1)=span \lbrace a_{4}, a_5\rbrace $ is included in $[\mathcal{A}_1, \mathcal{A}_1]_{\Lie}$ $=span \lbrace a_{3}, a_{4}, a_5 \rbrace $,

  and

$\mathcal{A}_7 = span \lbrace g_{1}, g_{2}, g_{3}, g_4, g_5 \rbrace$, with non-zero multiplications $[g_{1}, g_1]=g_{3}$, $[g_{1}, g_2]$ $=g_{4}$, $[g_2, g_1]= g_5$ and $[g_1, g_3]= g_5$, in which
$Z_{\Lie}(\mathcal{A}_7)=span \lbrace g_{4 }, g_5\rbrace$ is included in $ [\mathcal{A}_7, \mathcal{A}_7]_{\Lie}$ $=span \lbrace g_{3}, g_4, g_5\rbrace $.

We define $\omega: \frac{\mathcal{A}_1}{Z_{\Lie}(\mathcal{A}_1)}\longrightarrow \frac{\mathcal{A}_7}{Z_{\Lie}(\mathcal{A}_7)}$ by $ \omega (\bar{a_{1}})=\bar{g_{1}},~ \omega (\bar{a_{2}})=\bar{g_{2}},~ \omega (\bar{a_{3}})=\bar{g_{3}}, ~ \omega (\bar{a_{4}})=\bar{0}=\bar{g_{4}}, ~ \omega (\bar{a_{5}})=\bar{0}=\bar{g_{5}}$ and $\tau: [\mathcal{A}_1, \mathcal{A}_1]_{\Lie}\longrightarrow [\mathcal{A}_7, \mathcal{A}_7]_{\Lie}$ by $\tau (a_{3})=g_{3}, \tau (a_{4})=g_{4}+g_5, \tau (a_{5})=g_{5}$. Now it is easy to check that $\omega$ and $\tau$ are  isomorphisms making diagram (\ref{isoclinism}) commutative, hence $\mathcal{A}_1\sim \mathcal{A}_7$.

Note that the definition of $\tau$ reproduces the isomorphism given by Corollary \ref{corollary2}, namely  $ \tau_{\mid}:  Z_{\Lie}(\mathcal{A}_1)\cong  Z_{\Lie}(\mathcal{A}_7)$.

\item[{\it b)}]  Consider the following  four-dimensional non isomorphic {\Lie}-stem Leibniz algebras given in \cite{RefCaKu}:

 $\mathfrak{q}=span \lbrace e_{1}, e_{2}, e_{3}, e_4 \rbrace$, with non-zero multiplications $[e_{1}, e_4]= e_{1}$, $[e_{2}, e_4]$
 $=e_{2}$ and $[e_4, e_4]= e_{3}$ (class  ${\mathcal L}_{26}(\mu_2)$, with $\mu_2=1$, in \cite[Proposition 3.11]{RefCaKu}). Take the two-sided ideal $\mathfrak{m} = span \{e_1,e_2,e_3\}$ of $\mathfrak{q}$. Then $(\mathfrak{m},\mathfrak{q})$ is a {\Lie}-stem pair since
 $Z_{\Lie}(\mathfrak{m},\mathfrak{q})=span \lbrace e_{3 }\rbrace$ and $[\mathfrak{m},\mathfrak{q}]_{\Lie}=span \lbrace e_{1}, e_{2}, e_3 \rbrace$,

  and

 $\mathfrak{p}=span \lbrace a_{1}, a_{2}, a_{3}, a_4 \rbrace$, with non-zero multiplications $[a_{1}, a_4]= a_{2}$, $[a_{3}, a_4]$
 $=a_{3}$ and $[a_4, a_4]= a_{1}$ (class  ${\mathcal L}_{40}$ in \cite[Proposition 3.11]{RefCaKu}). Take the two-sided ideal $\mathfrak{n} = span \{a_1,a_2,a_3\}$ of $\mathfrak{p}$. Then $(\mathfrak{n},\mathfrak{p})$ is a {\Lie}-stem pair since
 $Z_{\Lie}(\mathfrak{n},\mathfrak{p})=span \lbrace a_{2}\rbrace$ and $[\mathfrak{n},\mathfrak{p}]_{\Lie}=span \lbrace a_{1}, a_{2}, a_3 \rbrace$.

Now we define the isomorphisms $\omega: \frac{\mathfrak{q}}{Z_{\Lie}(\mathfrak{m},\mathfrak{q})}\longrightarrow  \frac{\mathfrak{p}}{Z_{\Lie}(\mathfrak{n},\mathfrak{p})}$ by $\omega (\bar{e_{1}})=\bar{a_{1}},~ \omega (\bar{e_{2}})=\bar{a_{3}},~ \omega (\bar{e_4})= \bar{a_4} $ and $ \tau: [\mathfrak{m},\mathfrak{q}]_{\Lie}\longrightarrow [\mathfrak{n},\mathfrak{p}]_{\Lie}$ by $ \tau (e_{1})=a_{2}, \tau (e_{2})=a_{3}, \tau (e_{3})=a_{1}$.

Now it is easy to check that $\omega$ and $\tau$ are  isomorphism making diagram (\ref{isoclinism}) commutative,   hence $(\mathfrak{m},\mathfrak{q})$  and $(\mathfrak{n},\mathfrak{p})$ are {\Lie}-isoclinic.

\end{enumerate}
\end{example}

The aim of the rest of the paper is to find  conditions under what {\Lie}-isoclinism between two {\Lie}-stem Leibniz algebras implies their isomorphism.

\begin{lemma}\label{lemma4}
For any {\Lie}-stem Leibniz algebra $\mathfrak{m}$, $Z(\mathfrak{m}) = Z_{\Lie}(\mathfrak{m})$.
\end{lemma}
\begin{proof}
Since for every $y\in \mathfrak{m}$ and $ z\in  Z_{\Lie}(\mathfrak{m}) $, $[y, z]=0$, it follows that $[z, y]=[y, z]=0$ and therefore $Z_{\Lie}(\mathfrak{m}) \subseteq Z(\mathfrak{m})$.
A direct checking  shows that $Z(\mathfrak{m}) \subseteq Z_{\Lie}(\mathfrak{m})$.
\end{proof}

\begin{theorem}\label{theorem3}
Let $( \mathfrak{m}, \mathfrak{q})$ and  $( \mathfrak{n}, \mathfrak{p})$ be two {\Lie}-isoclinic pairs of finite dimensional complex Leibniz algebras such that:
\begin{enumerate}
\item[a)]  $\mathfrak{m}$ and $\mathfrak{n}$ are {\Lie}-stem Leibniz algebras.
\item[b)] For all  elements $ {m_1}, {m_2}\in {\mathfrak{m}}$ there exists $ \varepsilon_{12}\in \mathbb{C}$ such that $[{m_1}, {m_2}]={\varepsilon_{12}} [{m_2}, {m_1}]$.
 \end{enumerate}

Then $\mathfrak{m}$ and $\mathfrak{n}$ are isomorphic.
\end{theorem}
\begin{proof}
First of all, we claim that $( \mathfrak{m}, \mathfrak{q})$ and  $( \mathfrak{n}, \mathfrak{p})$ are {\Lie}-stem pairs. Indeed, $Z_{\Lie}(\mathfrak{m}, \mathfrak{q}) \subseteq Z_{\Lie}(\mathfrak{m}) \subseteq [\mathfrak{m},\mathfrak{m}]_{\Lie}\subseteq [\mathfrak{m},\mathfrak{q}]_{\Lie}$.

Owing to Proposition \ref{proposition3}, there exist two factor sets $f: \frac{\mathfrak{m}}{Z_{\Lie}(\mathfrak{m}, \mathfrak{q})} \times \frac{\mathfrak{m}}{Z_{\Lie}(\mathfrak{m}, \mathfrak{q})} \to Z_{\Lie}(\mathfrak{m}, \mathfrak{q})$ and $g: \frac{\mathfrak{n}}{Z_{\Lie}(\mathfrak{n}, \mathfrak{p})} \times \frac{\mathfrak{n}}{Z_{\Lie}(\mathfrak{n}, \mathfrak{p})} \to Z_{\Lie}(\mathfrak{n}, \mathfrak{p})$ such that $\mathfrak{n} \cong \mathfrak{m}_f$ and  $\mathfrak{m} \cong \mathfrak{n}_g$.  Let $(\omega, \tau)$ be the {\Lie}-isoclinism between $\mathfrak{m}_f$ and $\mathfrak{n}_g$ provided by Corollary \ref{corollary1}, then the following diagram is commutative,

\begin{equation*}
\begin{CD}
\frac{\mathfrak{m}_f}{Z_{\Lie}(\mathfrak{m}_f)}\times \frac{\mathfrak{m}_f}{Z_{\Lie}(\mathfrak{m}_f)}@>{C_1}>>[\mathfrak{m}_f, \mathfrak{m}_f]_{\Lie}\\
@VV{{\omega}\times \omega}V @VV{\tau}V\\
\frac{\mathfrak{n}_g}{Z_{\Lie}(\mathfrak{n}_g)}\times \frac{\mathfrak{n}_g}{Z_{\Lie}(\mathfrak{n}_g)}@>{C_2}>>[\mathfrak{n}_g, \mathfrak{n}_g]_{\Lie},
\end{CD}
\end{equation*}

We know that:
$$\omega ((z, \bar{m})+Z_{\Lie}(\mathfrak{m}_f))= \omega ((0, \bar{m}) +Z_{\Lie}(\mathfrak{m}_f))=(0, \omega_1(\bar{m}))+Z_{\Lie}(\mathfrak{n}_g)$$
where $\omega_1: \mathfrak{m}/ Z_{\Lie}(\mathfrak{m}, \mathfrak{q})\longrightarrow \mathfrak{n}/ Z_{\Lie}(\mathfrak{n}, \mathfrak{p})$ is an isomorphism.

On the other hand, from the following diagram
$$\begin{CD}
Z_{\Lie}^{\mathfrak{n}_g}\cong Z_{\Lie}(\mathfrak{n}, \mathfrak{p})\subseteq Z_{\Lie}(\mathfrak{n})\subseteq [\mathfrak{n}, \mathfrak{n}]_{\Lie}\cong [\mathfrak{m}_f, \mathfrak{m}_f]_{\Lie}\\
\quad \quad \quad  \quad \quad \quad \quad \quad \quad \quad \quad \quad \quad \quad \quad \quad \quad  @VV{\tau}V\\
\quad Z_{\Lie}^{\mathfrak{m}_f}\cong Z_{\Lie}(\mathfrak{m}, \mathfrak{q})\subseteq Z_{\Lie}(\mathfrak{m})\subseteq [\mathfrak{m}, \mathfrak{m}]_{\Lie}\cong [\mathfrak{n}_g, \mathfrak{n}_g]_{\Lie}\\
\end{CD}$$\\
and keeping in mind that $Z_{\Lie}(\mathfrak{m}, \mathfrak{q})$ is a two-sided ideal of $[\mathfrak{m}, \mathfrak{m}]_{\Lie}$, then from the proof of Theorem \ref{theorem2}, Corollary \ref{corollary2} and the following commutative diagram
\[ \xymatrix{
0 \ar[r] & Z_{\Lie}(\mathfrak{n}, \mathfrak{p})\cong Z_{\Lie}^{\mathfrak{n}_g} \ar[r] \ar@{-->}[d]^{\tau_{\mid}} & [\mathfrak{n}, \mathfrak{n}]_{\Lie} \ar[r] \ar[d]_{\wr}^{\tau}&
\frac{[\mathfrak{n}, \mathfrak{n}]_{\Lie}}{Z_{\Lie}(\mathfrak{n}, \mathfrak{p})} \ar[r] \ar[d]_{\wr}^{\bar{\tau}} & 0\\
0 \ar[r] & Z_{\Lie}(\mathfrak{m}, \mathfrak{q})\cong Z_{\Lie}^{\mathfrak{m}_f} \ar[r]  & [\mathfrak{m}, \mathfrak{m}]_{\Lie} \ar[r] &
\frac{ [\mathfrak{m}, \mathfrak{m}]_{\Lie}}{Z_{\Lie}(\mathfrak{m}, \mathfrak{q})}\ar[r]  & 0,
} \]
we conclude that $\tau_{\mid}:Z_{\Lie}^{\mathfrak{n}_g} \cong Z_{\Lie}^{\mathfrak{m}_f}$.

Now we define $\tau_{2}:Z_{\Lie}(\mathfrak{m}, \mathfrak{q})\longrightarrow Z_{\Lie}(\mathfrak{n}, \mathfrak{p})$  by $\tau (z,0)=(\tau_{2} (z), 0)$, for all $z\in   Z_{\Lie}(\mathfrak{m}, \mathfrak{q})$.

 Using Lemma \ref{lemma4}, for all $ (z_i, \bar{m_i})\in \mathfrak{m}_f $, $ i=1, 2 $,
\[
\begin{array}{l}
\tau \left( \left[(z_1, \bar{m_1}),(z_2, \bar{m_2})\right] +  \left[(z_2, \bar{m_2}),(z_1, \bar{m_1})\right] \right) =\\

\tau\Big( \left( [z_1, z_2]+ [z_1, \mu (\bar{m_2})]+ [\mu(\bar{m_1}), z_2] + f(\bar{m_1}, \bar{m_2}), [\bar{m_1}, \bar{m_2}]\right)+\\
\left( [z_2, z_1]+ [z_2, \mu (\bar{m_1})]+ [\mu(\bar{m_2}), z_1] + f(\bar{m_2}, \bar{m_1}), [\bar{m_2}, \bar{m_1}]\right)\Big)=\\

 \left( \tau_2( f(\bar{m_1}, \bar{m_2})+f(\bar{m_2}, \bar{m_1})), 0\right)+
\tau \left(0, [\bar{m_1}, \bar{m_2}]+[\bar{m_2}, \bar{m_1}]\right).
\end{array}
\]
where $\mu$ is as Proposition \ref{proposition4}. On the other hand,
\[
\begin{array}{l}
\tau \left( \left[(z_1, \bar{m_1}),(z_2, \bar{m_2})\right]+\left[(z_2, \bar{m_2}),(z_1, \bar{m_1})\right] \right)= \\
C_2(\omega ((z_1, \bar{m_1})+Z_{\Lie}(\mathfrak{m}_f)), \omega((z_2, \bar{m_2})+Z_{\Lie}(\mathfrak{m}_f))=\\
C_2((0, \omega_1 (\bar{m_1}))+Z_{\Lie}(\mathfrak{n}_g), (0, \omega_1 (\bar{m_2}))+Z_{\Lie}(\mathfrak{n}_g))=\\
\Big( g(\omega_1(\bar{m_1}), \omega_1(\bar{m_2}))+g(\omega_1(\bar{m_2}), \omega_1(\bar{m_1})), [\omega_1(\bar{m_1}), \omega_1(\bar{m_2})]+[\omega_1(\bar{m_2}), \omega_1(\bar{m_1})]\Big).\\
\end{array}
\]
Let $d\left(\omega_1 ([\bar{m_1}, \bar{m_2}]+[\bar{m_2}, \bar{m_1}]) \right)$ be the first component of $\tau \left(0, [\bar{m_1}, \bar{m_2}]+[\bar{m_2}, \bar{m_1}]\right)$ where $ d: \left[\frac{\mathfrak{n}}{Z_{\Lie}(\mathfrak{n}, \mathfrak{p})}, \frac{\mathfrak{n}}{Z_{\Lie}(\mathfrak{n}, \mathfrak{p})} \right]_{\Lie}\longrightarrow Z_{\Lie}(\mathfrak{n}, \mathfrak{p})$ is a linear map. Now the isomorphism $ \tau $ yields:
\[
\begin{array}{rl}
\tau_2 \left( f(\bar{m_1}, \bar{m_2})+f(\bar{m_2}, \bar{m_1})) + d(\omega_1 ([\bar{m_1}, \bar{m_2}]+[\bar{m_2}, \bar{m_1}])\right)&=\\
g(\omega_1(\bar{m_1}), \omega_1(\bar{m_2}))+g(\omega_1(\bar{m_2}), \omega_1(\bar{m_1})).&
\end{array}
\]
Applying assumption {\it b)} and Lemma \ref{lemma4}, we conclude that
\begin{equation}\label{best}
\tau_2 \left( f(\bar{m_1}, \bar{m_2})) + d(\omega_1 ([\bar{m_1}, \bar{m_2}])\right)=
g(\omega_1(\bar{m_1}), \omega_1(\bar{m_2}))
\end{equation}

Now we define $\lambda:\mathfrak{m}_f\longrightarrow \mathfrak{n}_g$ by $\lambda(z,\bar{m})=(\tau_2 (z)+d(\omega_1(\bar{m})), {\omega}_1(\bar{m}))$.
By Lemma \ref{lemma4}, for all $ (z_i, \bar{m_i})\in \mathfrak{m}_f $, $ i=1, 2 $, we get,
\begin{align*}
\lambda([(z_1, \bar{{m}_1}]), (z_2, \bar{{m}_2})])&=
\Big(\tau_2 {\big([z_1, z_2]+[\mu(\bar{{m}_1}), z_2]+[z_1, \mu(\bar{{m}_2})]+f(\bar{{m}_1}, \bar{{m}_2})\big)}\\
&+ d(\omega_1([\bar{{m}_1}, \bar{{m}_2}])), {{\omega}_1}([\bar{{m}_1}, \bar{{m}_2}])\Big)\\
& = \left( \tau_2 {\big( f(\bar{{m}_1}, \bar{{m}_2})\big)}+d(\omega_1([\bar{{m}_1}, \bar{{m}_2}])), \omega_1([\bar{{m}_1}, \bar{{m}_2}]) \right).
\end{align*}
On the other hand,
\begin{align*}
[\lambda(z_1, \bar{{m}_1}), \lambda(z_2, \bar{{m}_2})] &=
 \Big([\tau_2 (z_1)+d(\omega_1(\bar{{m}_1})), \tau_2 (z_2)+d(\omega_1(\bar{{m}_2}))]\\
& +[\tau_2 (z_1)+d(\omega_1(\bar{{m}_1})), \nu(\omega{(\bar{{m}_2})})]\\
&+[\nu(\omega_1{(\bar{{m}_1})}), \tau_2 (z_2)+d(\tau_1(\bar{{m}_2}))]\\
&+g(\omega_{1}(\bar{m_1}), \omega_{1}(\bar{m_2})), [\omega_{1}(\bar{m_1}), \omega_{1}(\bar{m_2})]\Big)\\
& = \left( g(\omega_{1}(\bar{m_1}), \omega_{1}(\bar{m_2})),[\omega_{1}(\bar{m_1}), \omega_{1}(\bar{m_2})] \right),
\end{align*}
 where $\nu$ is as in Proposition \ref{proposition4}. It now follows from equality \ref{best}, that $ \lambda $ is a homomorphism. The following diagram
\[ \xymatrix{
0 \ar[r] & Z_{\Lie}^{\mathfrak{n}_g}\cong Z_{\Lie}(\mathfrak{n}, \mathfrak{p}) \ar[r] \ar[d]^{{\tau_2}^{-1}} & \mathfrak{m}_f\cong \mathfrak{n} \ar@{-->}[d]^{\lambda}
 \ar[r]^{\pi_1~~~~} & \dfrac{\mathfrak{m}_f}{Z_{\Lie}^{\mathfrak{n}_g}}\cong\frac{\mathfrak{n}}{Z_{\Lie}(\mathfrak{n}, \mathfrak{p})} \ar[r] \ar[d]^{{\omega_1}^{-1}} & 0 \\
0 \ar[r] & Z_{\Lie}^{\mathfrak{m}_f}\cong Z_{\Lie}(\mathfrak{m}, \mathfrak{q}) \ar[r]  & \mathfrak{n}_g\cong \mathfrak{m}  \ar[r]^{\pi_2 ~~~~} & \dfrac{\mathfrak{n}_g}{Z_{\Lie}^{\mathfrak{m}_f}}\cong\frac{\mathfrak{m}}{Z_{\Lie}(\mathfrak{m}, \mathfrak{q})} \ar[r] & 0
} \]
is commutative, where $\omega_1$ and $\tau_2$ are the above isomorphisms. Now by the Short Five Lemma \cite{RefBB}, it follows that $\lambda$ is an isomorphism.
\end{proof}
The above Theorem still holds for finite dimensional stem pair of Lie algebras if we drop the assumptions
{\it a)} and {\it b)}, see \cite{RefMP} for more details.
\begin{example}
An example of non-Lie Leibniz algebra satisfying condition {\it b)} of Theorem \ref{theorem3} is the two-dimensional Leibniz algebra with basis $\{a_1, a_2 \}$ and the bracket operation given by $[a_2, a_2] = \lambda a_1, \lambda \in \mathbb{C} \setminus \mathbb{C}^2$ \cite{RefCu}. Here the parameter $\varepsilon = 1$.

It should be mention that, in Example \ref{example3} ({\it a}), Leibniz algebra $\mathcal{A}_7$ does not satisfy in the condition {\it b)} of Theorem \ref{theorem3}. 

Note that the parameter $\varepsilon$ in condition {\it b)} of Theorem \ref{theorem3} doesn't  require to be unique, that is  every pair of elements $\bar{m}_i, \bar{m}_j$, has a parameter $\varepsilon_{ij}$ such that $[\bar{m}_i, \bar{m}_j ]=\varepsilon_{ij} [\bar{m}_j, \bar{m}_i]$. An example of Leibniz algebra satisfying this condition is the four-dimensional Leibniz algebra with basis $\{a_1, a_2, a_3, a_4\}$ with bracket operation $[a_1,a_1]=a_3, [a_2,a_4]=- [a_4,a_2]=a_2; [a_4, a_4]=- 2 a_2$ (class ${\mathcal L}_{16}$ in \cite[Proposition 3.10]{RefCaKu}). Here the parameters are $\varepsilon_{11}=1, \varepsilon_{24}=-1, \varepsilon_{44}=1$.
\end{example}

In Lie algebra theory, it is well-known that the isoclinism between two finite dimensional stem Lie algebras implies 
isomorphism. For a deeper discussion, we refer the reader to \cite{RefM}.  In the following we provide some conditions to have analogous result.
\begin{corollary}
Let $ \mathfrak{q}$ and  $\mathfrak{p}$ be {\Lie}-isoclinic {\Lie}-stem of finite dimensional complex Leibniz algebras and for all elements $ {x_1}, {x_2}\in {\mathfrak{q}}$ there exist  $ \varepsilon_{12}\in \mathbb{C} $ such that $ [{x_1}, {x_2}]={\varepsilon_{12}} [{x_2}, {x_1}]$.
 Then $  \mathfrak{q}$ and  $ \mathfrak{p}$ are isomorphic.
\end{corollary}



\begin{thebibliography}{0}
\bibitem{RefBC}Biyogmam, R. G.,  Casas, J. M.:  On Lie-isoclinic Leibniz algebras, J. Algebra in press (2017). http://dx.doi.org/10.1016/j.algebra.2017.01.034.
\bibitem{RefBB} Borceaux, F., Bourn, D.:Mal'cev, protomodular, homological and semi-abelian categories, Math. Appl., vol.  566, Kluwer Acad. Publ. (2004).
\bibitem{RefCaKu}Ca\~{n}ete, E. M., Khudoyberdiyev, A. Kh.: The classification of 4-dimensional Leibniz algebras, Linear Algebra Appl. 439 (1), 273--288 (2013).
\bibitem{RefCIL}Casas, J. M., Insua, M. A., Ladra, M., Ladra, S.: An algorithm for the classification of 3-dimensional complex Leibniz algebras, Linear Algebra Appl. 436, 3747--3756 (2012).
\bibitem{RefCKh} Casas, J. M.,  Khmaladze, E.: On Lie-central extensions of Leibniz algebras, Rev. R. Acad. Cien. Serie A. Mat. 111 (1), 39--56 (2017).
\bibitem{RefCE}Chevalley, C., Eilenberg, S.: Cohomology theory of Lie groups and Lie algebras,  Trans Amer. Math. 63, 85--124 (1948). 
\bibitem{RefCu}Cuvier, C.: Alg\`{e}bres de Leibnitz, d\'{e}finitions, propri\'{e}t\'{e}s, Ann. Sci. \'{E}col. Norm. Sup. 27 (4), 1--45 (1994).
\bibitem{RefD}Demir, I.: Classification of 5-Dimensional Complex Nilpotent Leibniz Algebras, Ph.D. Thesis, North Carolina State University, Raleigh, NC, USA, (2016).
\bibitem{RefH} Hall, P.:  The classification of prime-power groups, J. Reine Angew. Math. 182, 130--141 (1940). 
\bibitem{RefJK}Janelidze, G., Kelly, G. M.: Galois theory and a general notion of central extension, J. Pure Appl. Algebra 97, 135--161 (1994).
\bibitem{RefJMT}Janelidze, G., M\'arki, L., Tholen, W.: Semi-abelian categories, J. Pure Appl. Algebra 168, 367--386 (2002).
\bibitem{RefLSW}Liu, J., Sheng, Y., Wang, Q.: On non-abelian extensions of Leibniz  algebras, Comm. Algebra 46, 574--587 (2018).
\bibitem{RefL}Loday, J.-L.: Cyclic Homology, Grund. Math. Wiss. 301, 2nd edition. Springer-Verlag Berlin etc. (1998).
\bibitem{RefLP}Loday, J.-L., Pirashvili, T.:  Universal  enveloping  algebras of Leibniz  algebras and  (co)homology, Math. Ann. 296, 139--158(1993).
\bibitem{RefMP}Moghaddam, M. R. R.,  Parvaneh, F.: On the isoclinism of a pair of Lie algebras and factor sets, Asian-Eur. J. Math. 2 (2), 213--225 (2009).
\bibitem{RefMSR}Mohammadzadeh, H.,  Salemkar, A. R., Riyahi, Z.: Isoclinic extensions of Lie algebras, Turkish J. Math. 37 (4), 598--606 (2013).
\bibitem{RefM}Moneyhun, K.: Isoclinisms in Lie algebras, Algebras Groups Goem. 11, 9--22 (1994).
\bibitem{RefT}Tappe, J.:  On isoclinic groups, Math. Zeits. 148, 147--153 (1976).
\bibitem{RefW}Weichsel, P. W.: On isoclinism. J. London Math. Soc. 38, 63--65 (1963).
\end{thebibliography}
\end{document}